\documentclass[oneside,american,english]{amsart}
\usepackage{libertine-type1}
\usepackage{biolinum-type1}
\usepackage{libertineMono-type1}
\usepackage{eulervm}
\usepackage[T1]{fontenc}
\usepackage[latin9]{inputenc}
\usepackage{verbatim}
\usepackage{textcomp}
\usepackage{amstext}
\usepackage{amsthm}
\usepackage{amssymb}
\usepackage[all]{xy}
\PassOptionsToPackage{normalem}{ulem}
\usepackage{ulem}

\makeatletter
\numberwithin{equation}{section}
\numberwithin{figure}{section}
  \theoremstyle{plain}
  \newtheorem*{thm*}{\protect\theoremname}
\theoremstyle{plain}
\newtheorem{thm}{\protect\theoremname}[section]
  \theoremstyle{remark}
  \newtheorem{rem}[thm]{\protect\remarkname}
  \theoremstyle{definition}
  \newtheorem{defn}[thm]{\protect\definitionname}
  \theoremstyle{plain}
  \newtheorem{fact}[thm]{\protect\factname}
  \theoremstyle{plain}
  \newtheorem{cor}[thm]{\protect\corollaryname}
  \theoremstyle{plain}
  \newtheorem{lem}[thm]{\protect\lemmaname}
  \theoremstyle{remark}
  \newtheorem*{rem*}{\protect\remarkname}
  \theoremstyle{plain}
  \newtheorem{prop}[thm]{\protect\propositionname}

\usepackage{a4wide}
\makeatother

\def\Ind#1#2{#1\setbox0=\hbox{$#1x$}\kern\wd0\hbox to 0pt{\hss$#1\mid$\hss}
\lower.9\ht0\hbox to 0pt{\hss$#1\smile$\hss}\kern\wd0}

\def\Notind#1#2{#1\setbox0=\hbox{$#1x$}\kern\wd0\hbox to 0pt{\mathchardef
\nn="3236\hss$#1\nn$\kern1.4\wd0\hss}\hbox to
0pt{\hss$#1\mid$\hss}\lower.9\ht0
\hbox to 0pt{\hss$#1\smile$\hss}\kern\wd0}

\makeatother

\usepackage{babel}
  \addto\captionsamerican{\renewcommand{\corollaryname}{Corollary}}
  \addto\captionsamerican{\renewcommand{\definitionname}{Definition}}
  \addto\captionsamerican{\renewcommand{\factname}{Fact}}
  \addto\captionsamerican{\renewcommand{\lemmaname}{Lemma}}
  \addto\captionsamerican{\renewcommand{\propositionname}{Proposition}}
  \addto\captionsamerican{\renewcommand{\remarkname}{Remark}}
  \addto\captionsamerican{\renewcommand{\theoremname}{Theorem}}
  \addto\captionsenglish{\renewcommand{\corollaryname}{Corollary}}
  \addto\captionsenglish{\renewcommand{\definitionname}{Definition}}
  \addto\captionsenglish{\renewcommand{\factname}{Fact}}
  \addto\captionsenglish{\renewcommand{\lemmaname}{Lemma}}
  \addto\captionsenglish{\renewcommand{\propositionname}{Proposition}}
  \addto\captionsenglish{\renewcommand{\remarkname}{Remark}}
  \addto\captionsenglish{\renewcommand{\theoremname}{Theorem}}
  \providecommand{\corollaryname}{Corollary}
  \providecommand{\definitionname}{Definition}
  \providecommand{\factname}{Fact}
  \providecommand{\lemmaname}{Lemma}
  \providecommand{\propositionname}{Proposition}
  \providecommand{\remarkname}{Remark}
  \providecommand{\theoremname}{Theorem}
\providecommand{\theoremname}{Theorem}

\begin{document}
\global\long\def\Rr{\mathbb{R}}
\global\long\def\dim#1{\mathrm{d}\left(#1\right)}
\global\long\def\cl#1{\mathrm{cl}\left(#1\right)}
\global\long\def\gcl#1{\mathrm{gcl}\left(#1\right)}
\global\long\def\Aut{\mathrm{Aut}}
\global\long\def\Autf{\mathrm{Autf}}
\global\long\def\acleq{\mbox{\ensuremath{\mathrm{acl}}\ensuremath{\ensuremath{^{eq}}}}}
\global\long\def\bdd{\mathrm{Bdd}}

\title{The small index property for homogeneous models in AEC's}
\authors{
\author{Zaniar Ghadernezhad}
\address{ School of Mathematics\\ Institute for Research in Fundamental Sciences (IPM)\\ P.O. Box 19395-5746\\ Tehran, Iran.}
\email{zaniar.gh@gmail.com}
\author{Andr\'es Villaveces}
\address{ Departamento de Matem\'aticas, Universidad Nacional de Colombia, Bogot\'a 111321, Colombia}
\email{avillavecesn@unal.edu.co}}
\thanks{This research was made possible partially by Colciencias grant M\'etodos de Estabilidad en Clases No Estables. The first author's research was also partially supported by the Iran National Science Foundation (INSF)}
\maketitle

\tableofcontents{}

\section{Introduction }

The role of automorphism groups of mathematical structures and their
connections to their theories and \emph{bi-interpretation} classes
has been an active research area in model theory in the past two decades.
The main theme has been to investigate what kind of information from
the structure can be recovered from its group of automorphisms (the
\emph{reconstruction problem}). The automorphism groups are also topological
groups, in a natural way (pointwise convergence topology). The reconstruction
problem is therefore linked to both the purely algebraic aspects of
the groups $Aut(M)$, as well as to their topological aspects. The
crucial property linking these two aspects is the ``small index property'',
the center of study of this paper. A structure $M$ of cardinality
$\kappa$ has the \emph{small index property} (SIP) if every subgroup
$H\leq Aut(M)$ of small index (that is, such that $\left[Aut\left(M\right):H\right]<\left|Aut\left(M\right)\right|$)
is open in $Aut(M)$.

Although in principle the SIP is posed in terms of a (first-order)
structure $M$, previous results have relied heavily on the \emph{first
order theory }of $M$ and its bi-interpretability class. The structural
properties revealed by $Aut(M)$ depend somehow more on the theory
of $M$, on the structural properties \emph{around} $M$, than on
$M$ taken in isolation. In this article, we look for the first time
at the situation of a structure $M$ that may fail to be saturated
yet still have a good reconstruction problem, provided by other structural
properties: 
\begin{itemize}
\item homogeneity,
\item being inside an abstract elementary class (AEC) and
\item having a strong amalgamation property (for strong embeddings in the
abstract elementary class, \emph{and} for automorphisms).
\end{itemize}
These properties are isolated by carefully checking the case studied
before by Lascar and Shelah \cite{MR1204064}, where they prove the
SIP for uncountable, saturated structures, relying on their first
order theories. In this paper we show that their techniques may be
adapted to our more general setup: saturation is weakened to homogeneity,
and we isolate a strong notion of amalgamation in AECs that suffices
to replace the technical aspects of the proof from \cite{MR1204064}
that depend on the first order theory of the structure. So, our result
shows that the SIP may be obtained for homogeneous structures, provided
they are placed in an AEC that has enough structural properties\footnote{Classical model theory deals with mathematical structures using tools
of first-order logic; the first-order setting provides many technical
tools to study them. However, the first-order setting has some limitations:
many classes of mathematical structures are not first-order axiomatizable,
and even when they are, there are situations in which a ``non-elementary''
model theoretical analysis results in better regularity properties.
Shelah in \cite{MR1033034} introduced the notion of \emph{Abstract
Elementary Classes} (AECs), where logic and syntax are set aside and
the elements of the class are axiomatized using an abstract notion
of ``strong'' embedding (important examples of these classes include
Zilber's pseudo-exponentiation, various pseudo-analytic expansions
of the complex numbers and several instances of covers - more recently,
modular functions have been successfully analyzed using these tools
- see \cite{Harris}). In this general setting, we do not necessarily
need to consider the elements of the class in the framework of the
first-order logic and work with formulas. However, many important
concepts of model theory such as types, forking and other independence
notions have been successfully studied in the AEC context.}.

An important step in reconstruction problems is therefore determining
whether a structure has the SIP: this property is a key piece in the
recovery of the topological structure of the automorphism group from
its pure group structure. The SIP has been proved for a number of
countable first-order structures. With the pointwise convergence topology,
in the first-order case, the automorphism groups of countable structures
are actually \emph{Polish} groups (see \cite{MR1066691} for more
details). Indeed, the automorphism groups of countable structures
being Polish is a very useful fact that can provide many technical
tools to prove properties like SIP. For the uncountable first-order
structures, we do not have access to such tools, however still nice
properties can be proven: In the uncountable case, with some cardinal
restrictions, Lascar and Shelah in \cite{MR1204064} proved that the
automorphism group of an uncountable saturated model has SIP.

In this paper, we prove the following theorem (Theorem \ref{thm:main}
in this paper, see section \ref{sec:main-thm}):
\begin{thm*}
Let $M$ be a homogeneous model in an abstract elementary class $\left(K,\preccurlyeq_{\mathcal{K}}\right)$
such that \textup{$\left|M\right|=\kappa>\mbox{LS}\left(K\right)$
}\textup{\emph{and}}\textup{ $\kappa^{<\kappa}=\kappa$. }\textup{\emph{Furthermore,
assume that }}\emph{$\mathcal{K}^{<}\left(M\right)$ }is a\emph{ }strong
amalgamation clas\textup{\emph{s. }}Consider the group\textup{\emph{
$Aut(M)$}}\textup{ }with the topology given by $\mathfrak{T}^{\mbox{cl}}$,
and let $H\leqslant\Aut\left(M\right)$ be such that $\left[\Aut\left(M\right):H\right]\leq\kappa$.
\uline{Then}, $H$ is an open subgroup of $Aut(M)$; i.e., there
exists $A\in\mathcal{K}^{<}\left(M\right)$ such that $\Aut_{A}\left(M\right)\leqslant H$.

In other words, $\left(\Aut\left(M\right),\mathfrak{T}^{\mbox{cl}}\right)$
has the small index property. 
\end{thm*}
This theorem provides a purely algebraic framework for a model to
have the small index property; we thereby transfer in a rather sharp
way the Lascar-Shelah setting \cite{MR1204064}, where a similar result
is proved for saturated uncountable models of a first order
theory.

In the last two sections, we study some examples and applications:
we prove the SIP for the Zilber field (see \cite{MR2167585,LH,BHHKK}),
we study the SIP for the $j$-mapping (modular invariant) under the
light of recent results in its model theory, and for various other
$L_{\omega_{1},\omega}$-axiomatizable structures.

The authors thank the referee for several questions and remarks that
have helped improve the presentation of our results.

\section{Setting}

\label{sec:setting}

We fix two infinite cardinals $\lambda,\kappa$ such that $\lambda<\kappa$
and $\kappa^{<\kappa}=\kappa$. We work within a fixed AEC $\mathcal{K}=\left(K,\prec_{\mathcal{K}}\right)$
with $\mbox{LS}\left(\mathcal{K}\right)\leq\lambda$. We now provide
some notation and definitions.

\subsection{Notation}
\begin{itemize}
\item $\mbox{Iso}\left(N_{1};N_{2}\right)$ denotes the set of all isomorphisms
from $N_{1}$ onto $N_{2}$, for $N_{1},N_{2}\in K$.
\item $\mathcal{K}^{<}\left(M\right):=\left\{ N:N\ensuremath{\preccurlyeq_{\mathcal{K}}}M,\left|N\right|<\kappa\right\} $.
Clearly, this is not empty, as $LS(\mathcal{K})\leq\lambda<\kappa$.
\item For $N_{1},N_{2}\in\mathcal{K^{<}}\left(M\right)$ and $f\in\mbox{Iso}\left(N_{1};N_{2}\right)$
we define $\mathcal{O}_{f}^{M}:=\left\{ g\in\Aut\left(M\right):f\leqslant g\right\} $,
the set of all automorphisms of $M$ extending $f$. Note that $\mathcal{O}_{f}^{M}\neq\emptyset$
means that $f$ can be extended to an automorphism of $M$.
\item As usual, we denote by $g^{\alpha}$ the map $\alpha^{-1}\circ g\circ\alpha$.
\item Also, $Aut_{N}(M)$ denotes the pointwise stabilizer of $N$ (the
subgroup of automorphisms of $M$ that fix $N$ pointwise) and $Aut_{\left\{ N\right\} }(M)$
denotes the setwise stabilizer of $N$, the subgroup of automorphisms
of $M$ that fix $N$ setwise.
\end{itemize}
For the rest of the paper we fix a homogeneous model $M\in K$ of
size $\kappa$ (we use ``homogeneous'' in the following precise
sense: all isomorphisms between small strong substructures of $M$,
$f:N_{1}\to N_{2}$ for $N_{1},N_{2}\in\mathcal{K}^{<}(M)$, can be
extended to automorphisms of $M$). \emph{All our results} from here
on refer to properties of $M$ within the AEC $\mathcal{K}$.

In our context, ``small'' (subset or submodel) means ``of cardinality
$<\kappa$''.

We may even say that our focus of study is the part of $\mathcal{K}$
``below $M$'', i.e. the class of small $\mathcal{K}$-elementary
submodels of $M$. In this sense the actual arena of our results is
the class $\left(\mathcal{K}^{<}(M),\prec_{\mathcal{K}}\right)$.
This is almost an AEC; the only property of an AEC this fails to satisfy
is the unions axiom. In recent years, weak AECs have been studied;
this class is one of them.

Our notion of homogeneity means precisely that $\mathcal{O}_{f}^{M}\neq\emptyset$
for all $f\in\mbox{Iso}\left(N_{1};N_{2}\right)$ with $N_{1},N_{2}\in\mathcal{K^{<}}\left(M\right)$
when we translate the concept to this notation.
\begin{rem}
Let $N\in\mathcal{K^{<}}\left(M\right)$. Then the restriction map
$\pi_{N}:\Aut_{\left\{ N\right\} }\left(M\right)\rightarrow\Aut\left(N\right)$
is surjective: if $f\in Aut(N)$ then there is $f'\in Aut(M)$ such
that $f=f'\upharpoonright N=N$.
\end{rem}
\begin{defn}
For $A\subseteq M$, let $\mbox{cl}^{M}\left(A\right):=\bigcap\left\{ N\in\mathcal{K}\mid A\subseteq N\preccurlyeq_{\mathcal{K}}M\right\} $.
\end{defn}
\begin{rem}
This notion of closure appeared in \cite{MR2444267} in connection
with the study of the failure of tameness. It is worth noting here
that with this notion of closure, unlike in other related constructions,
closures \emph{do not necessarily belong} to the class $\mathcal{K}$.
Moreover, note that if $A\subseteq M$ and $\mbox{cl}^{M}\left(A\right)\neq M$,
then for every $m\in M\backslash\mbox{cl}^{M}\left(A\right)$ there
is $N\preccurlyeq_{\mathcal{K}}M$ such that $m\notin N$ and $\mbox{cl}^{M}\left(A\right)\subseteq N$.
\end{rem}
It is easy to see that the following holds.
\begin{enumerate}
\item $\mbox{cl}^{M}\left(N\right)=N$ for all $N\in K$ with $N\preccurlyeq_{\mathcal{K}}M$.
\item $\mbox{cl}^{M}\left(\mbox{cl}^{M}\left(A\right)\right)=\mbox{cl}^{M}\left(A\right)$
for all $A\subseteq M$.
\end{enumerate}
\begin{rem}
The definition above of $\mbox{cl}^{M}$ generalizes the first order
definition of the algebraic closure \emph{acl} with $\preccurlyeq_{\mathcal{K}}$
as elementary submodel.
\end{rem}
Let $\mathcal{C}:=\left\{ \mbox{cl}^{M}\left(A\right):A\subseteq M\mbox{ such that }\left|A\right|<\kappa\right\} $.
\begin{fact}
\begin{enumerate}
\item Suppose $A\subset M$, then $\left|\mbox{cl}^{M}\left(A\right)\right|\leq\left|A\right|+LS\left(K\right)$.
\item Small strong submodels of $M$ are closed i.e. $\mathcal{K}^{<}\left(M\right)\subseteq\mathcal{C}$.
\item If $A,B\in\mathcal{C}$, then $A\cap B\in\mathcal{C}$. 
\item If $A,B\in\mathcal{C}$, then there exists $N\in\mathcal{\mathcal{K}^{<}\left(M\right)}$
such that $A\cup B\subseteq N$.
\end{enumerate}
\end{fact}
\begin{proof}
(1) and (3) are immediate. (2) and (4) follow from the fact that $\left|M\right|=\kappa>\mbox{LS}\left(K\right)$.
\end{proof}
We extend the two notations that we have defined before from elements
of $\mathcal{K}^{<}(M)$ to all closed sets.
\begin{defn}
Let $X,Y\in\mathcal{C}$. Define 
\[
\mbox{Iso}\left(X;Y\right):=\left\{ f\upharpoonright X:f\in Iso\left(N_{1},N_{2}\right)\text{ where \ensuremath{N_{1},N_{2}\in\mathcal{K}^{<}\left(M\right),X\subseteq N_{1},Y\subseteq N_{2}\text{ and }f\left[X\right]=Y}}\right\} 
\]
 and $\Aut\left(X\right):=\mbox{Iso}\left(X;X\right)$.
\end{defn}
\begin{cor}
The set $\left\{ \Aut_{X}\left(M\right):X\in\mathcal{C}\right\} $
forms a basis of open neighborhoods of identity.
\end{cor}
Let $\mathfrak{T}^{\mbox{cl}}$ be the topology that is generated
by the cosets of stabilizers $Aut_X(M)$ for $X\in\mathcal{C}$. It
is clear that $\left(\Aut\left(M\right),\mathfrak{T}^{\mbox{cl}}\right)$
is a topological group. This is the setup for the proof of our main
theorem (see section $\ref{sec:main-thm}$).
\begin{rem}
Similarly to the countable case in first-order logic we have the following
property: Suppose $M$ and $N$ are two models of the same uncountable
cardinality and we topologize them in the same way. Suppose $\alpha:\Aut\left(M\right)\rightarrow\Aut\left(N\right)$
is an embedding as abstract groups. Then the $SIP$ for $M$ implies
that $\alpha$ is continuous.
\end{rem}

\subsection{Some properties of the class of closed sets $\mathcal{C}$}

In this subsection we do not require the assumption that $M$ is homogeneous. 
\begin{lem}
\label{lem:ep} Let $X,Y\in\mathcal{C}$ and $f\in\mbox{Iso}\left(X;Y\right)$
be such that $\mathcal{O}_{f}^{M}\neq\emptyset$. Then, there exists
$Z\in\mathcal{C}$ and $f'\in\Aut\left(Z\right)$ such that $X\cup Y\subseteq Z$
, $f\leqslant f'$ and $\mathcal{O}_{f'}^{M}\neq\emptyset$.
\end{lem}
\begin{proof}
Let $\sigma\in\mathcal{O}_{f}^{M}$; then $\sigma\in Aut(M)$ and
$\sigma\restriction X=f$. Since $\mbox{LS}\left(K\right)<\kappa$,
there exists $N\in\mathcal{K}^{<}\left(M\right)$ such that $X\cup Y\subseteq N$.
Let $X_{1}:=N\cup\sigma\left[N\right]\cup\sigma^{-1}\left[N\right]$
and let $\lambda=\left|X_{1}\right|$. It is clear that $\lambda<\kappa$,
hence there exists $N_{1}\in\mathcal{K}^{<}\left(M\right)$ of cardinality
$\lambda$ such that $X_{1}\subseteq N_{1}$. Inductively, define
$X_{n}:=N_{n-1}\cup\sigma\left[N_{n-1}\right]\cup\sigma^{-1}\left[N_{n-1}\right]$
and $N_{n}\in\mathcal{K}^{<}\left(M\right)$ such that $X_{n}\subseteq N_{n}$
and $\left|N_{n}\right|=\lambda$ for all $n\in\mathbb{N}$ (always
possible since $\mbox{LS}\left(K\right)<\kappa$). We then have a
chain $N_{i}\prec_{\mathcal{K}}N_{j}$ for $i\leq j\in\mathbb{N}$.
Let $N^{*}:=\bigcup_{n\in\mathbb{N}}N_{n}$. Then clearly $\left|N^{*}\right|=\lambda$
and $N^{*}\in\mathcal{K}^{<}\left(M\right)$. Note that $\sigma\upharpoonright N^{*}\in\Aut\left(N^{*}\right)$.
Let $Z:=N^{*}$ and $f':=\sigma\upharpoonright N^{*}$.
\end{proof}
\begin{cor}
Let $\sigma\in\Aut\left(M\right)$ and assume $\lambda$ is a cardinal
with $LS\left(K\right)\leq\lambda<\kappa$, then there exists $N\in\mathcal{K}^{<}\left(M\right)$
with $\left|N\right|=\lambda$ such that $\sigma\upharpoonright N\in\Aut\left(N\right)$.
\end{cor}
\begin{proof}
We use the same argument as the proof of Lemma \ref{lem:ep}: start
with $N_{0}$ and build up a $\prec_K$-increasing chain of models
$N_n$ such that $N_{n}\supseteq X_{n}:=N_{n-1}\cup\sigma\left[N_{n-1}\right]\cup\sigma^{-1}\left[N_{n-1}\right]$.
\end{proof}
\begin{rem}
So far, no special properties of $\kappa$ as a cardinality have been
used. For the next corollary we use $\kappa^{<\omega}=\kappa$.
\end{rem}
\begin{cor}
\label{cor:ext} Let $I$ be an index set with $\left|I\right|<\kappa$.
Assume $f_{i}\in Iso\left(X_{i};Y_{i}\right)$ with $X_{i},Y_{i}\in\mathcal{C}$
are such that $\mathcal{O}_{f_{i}}^{M}\not=\emptyset$, for $i\in I$.
Then there exists $Z\in\mathcal{C}$ and $f'_{i}\in\Aut_{\left\{ Z\right\} }\left(M\right)$
such that $\bigcup_{i\in I}X_{i}\cup\bigcup_{i\in I}Y_{i}\subseteq Z$,
$f_{i}\leqslant f'_{i}$ and $\mathcal{O}_{f_{i}'}^{M}\neq\emptyset$
for all $i\in I$.
\end{cor}
\begin{proof}
Use the facts that $\mbox{LS}\left(K\right)<\kappa$ and $\kappa^{<\kappa}=\kappa$
and repeat the argument in the proof of Lemma \ref{lem:ep}, by closing
the sets under all the $f_{i}'$ and $f_{i}'^{-1}$s for some $f_{i}'\in\mathcal{O}_{f}^{M}$
with $i\in I$ in each step.
\end{proof}
\begin{cor}
\label{cor:ep} Let $I$ be an index set with $\left|I\right|=\lambda$
such that $LS\left(K\right)\leq\lambda<\kappa$. Assume $\sigma_{i}\in\Aut\left(M\right)$
with $i\in I$. Then there exists $N\in\mathcal{K}^{<}\left(M\right)$
with $\left|N\right|=\lambda$ such that $\sigma{}_{i}\upharpoonright N\in\Aut\left(N\right)$
for all $i\in I$. 
\end{cor}
\begin{proof}
As in the proof of Lemma \ref{lem:ep}.
\end{proof}

\section{Generic sequences of automorphisms }

This section sets up the tools for the main proof: generic sequences
of automorphisms and strong amalgamation bases. We provide the basic
notions (the first one is adapted from \cite{MR1204064} to our context,
the second one is new) and then derive the existence of large families
of generic sequences of automorphisms.
\begin{defn}
\label{def:comp} Suppose $M_{0},M_{1},M_{2}\in\mathcal{K}^{<}\left(M\right)$
and $M_{0}\prec_{\mathcal{K}}M_{i}$ for $i=1,2$. Let $\gamma\text{\textsuperscript{0}}$
be a sequence of automorphisms of $M_{0}$. Assume $\gamma^{1}$ and
$\gamma^{2}$ are sequences of automorphisms of $M_{1}$ and $M_{2}$;
respectively, both extending $\gamma^{0}$. We say $\gamma\text{\textonesuperior}$
and $\gamma\text{\texttwosuperior}$ are \emph{compatible} over $M_{0}$
if there exist $M_{3}\in\mathcal{K}^{<}\left(M\right)$ and $\alpha_{1},\alpha_{2}\in\Aut_{M_{0}}\left(M\right)$
such that $\alpha_{1}\left[M_{1}\right],\alpha_{2}\left[M_{2}\right]\prec_{\mathcal{K}}M_{3}$
and $\gamma^{3}\in\Aut\left(M_{3}\right)$ extends both $\alpha_{1}\circ\gamma^{1}\circ\alpha_{1}^{-1}\restriction\alpha_{1}\left[M_{1}\right]$
and $\alpha_{2}\circ\gamma^{2}\circ\alpha_{2}^{-1}\restriction\alpha_{2}\left[M_{2}\right]$
(this is equivalent to requiring that $\gamma^{3}\circ\alpha_{1}\restriction N_{1}=\alpha_{1}\circ\gamma^{1}$
and $\gamma^{3}\circ\alpha_{2}\restriction N_{2}=\alpha_{2}\circ\gamma^{2}$). 
\end{defn}
Notice that in the definition above for a sequence of automorphisms
we allow repetitions of automorphisms. 

We now define a central notion in the proof of the main theorem: generic
sequences of automorphisms (and the particular case of a generic automorphism).
\begin{defn}
\label{def:gen} Suppose $I$ is an index set with $\left|I\right|<\kappa$
and let $\gamma=\left(g_{i}:i\in I\right)$ be a sequence of automorphisms
of $M$. We say $\gamma$ is a \emph{generic sequence of automorphisms
}if whenever $N\in\mathcal{\mathcal{K}}^{<}\left(M\right)$ is such
that $\gamma\upharpoonright N$ is a sequence of automorphisms of
$N$ and $N_{1}\in\mathcal{K}^{<}\left(M\right)$ is such that $N\prec_{\mathcal{K}}N_{1}$
and $\theta=\left(t_{i}:i\in I\right)$ is a sequence of automorphisms
of $N_{1}$ extending $\gamma\upharpoonright N$ (i.e. $g_{i}\upharpoonright N\leqslant t_{i}$
for all $i\in I$), \uline{then} if $\gamma\upharpoonright N'$
and $\theta$ are compatible over $N$ for some $N'\in\mathcal{K}^{<}\left(M\right)$,
there exists $\alpha\in\Aut_{N}\left(M\right)$ such that $\gamma$
extends $\alpha\circ\theta\circ\alpha^{-1}$ (or equivalently $\gamma^{\alpha}:=\left(g_{i}^{\alpha}:i\in I\right)$
extends $\theta$).
\end{defn}
We abuse language by saying that ``$g$ is a \emph{generic automorphism}''
when $\gamma=\left(g\right)$ is a constant generic sequence of automorphisms. 
\begin{rem}
Suppose $I$ is an index set with $\left|I\right|<\kappa$ and let
$\gamma=\left(g_{i}:i\in I\right)$ be generic. Then $\gamma^{\alpha}$
is generic for $\alpha\in\Aut\left(M\right)$. 
\end{rem}
\begin{lem}
Let $M_{0}\in\mathcal{K}^{<}\left(M\right)$ and $I$ be an index
set with cardinality less than $\kappa$. Suppose $\gamma=\left(g_{i}:i\in I\right)$
and $\theta=\left(t_{i}:i\in I\right)$ are two generic families of
automorphisms of $M$ such that $\gamma\upharpoonright M_{0}=\theta\upharpoonright M_{0}\in\Aut\left(M_{0}\right)$.
Then there exists $\alpha\in\Aut_{M_{0}}\left(M\right)$ such that
$\theta^{\alpha}=\gamma$.
\end{lem}
\begin{proof}
Similar to the argument in \cite{MR1204064}, we use back and forth.
Since $\theta$ and $\gamma$ are \emph{cofinal} in $\mathcal{K}^{<}\left(M\right)$
(see Corollary \ref{cor:ep}), we claim that we can build a chain
of models $N_{j}^{\theta}\in\mathcal{K}^{<}\left(M\right)$ and $N_{j}^{\gamma}\in\mathcal{K}^{<}\left(M\right)$
for $j<\kappa$, such that $\bigcup_{j<\kappa}N_{j}^{\theta}=\bigcup_{j<\kappa}N_{j}^{\gamma}=M$;
and $\theta\upharpoonright N_{j}^{\theta}\in\Aut\left(N_{j}^{\theta}\right)$
and $\gamma\upharpoonright N_{j}^{\gamma}\in\Aut\left(N_{j}^{\gamma}\right)$
and build partial isomorphisms $\alpha_{i}$ for $i<\kappa$ using
the back and forth such that $\alpha:=\bigcup_{j<\kappa}\alpha_{j}$
is the desired automorphism. Let $\alpha_{0}=\mbox{id}_{M_{0}}$ and
$N_{0}^{\theta}=N_{0}^{\gamma}=M_{0}$. If $i$ is limit ordinal,
let $N_{i}^{\theta}:=\bigcup_{j<i}N_{j}^{\theta}$, $N_{i}^{\gamma}:=\bigcup_{j<i}N_{j}^{\gamma}$
and $\alpha_{i}:=\bigcup_{j<i}\alpha_{j}$. Suppose $i=j+1$, without
loss of generality we do the ``forth'' step. Let $N_{i}^{\theta}$
be an element in $\mathcal{K}^{<}\left(M\right)$ such that $N_{j}^{\theta}\preccurlyeq_{\mathcal{K}}N_{i}^{\theta}$
and $\theta\upharpoonright N_{i}^{\theta}\in\Aut\left(N_{i}^{\theta}\right)$;
note that existence of such a model follows from Corollary \ref{cor:ep}
or proof of Lemma \ref{lem:ep}. %
{} Since $\gamma$ is generic, then there exists $\beta\in\Aut_{\alpha'_{j}\left[N_{j}^{\theta}\right]}\left(M\right)$
such that $\gamma$ extends $\beta\circ\theta\circ\beta^{-1}$ and
$\alpha'_{j}\in\mathcal{O}_{\alpha_{j}}^{M}$. Let $\alpha_{i}:=\alpha'_{j}\circ\beta\upharpoonright N_{i}^{\theta}$.
Let then $\alpha:=\bigcup\alpha_{j}$.
\end{proof}
It is an interesting question to determine sufficient conditions for
a generic sequence of automorphisms to exist. We now provide one condition
that will guarantee precisely that.
\begin{defn}
\label{Strong-Amalgamation-Base} Assume $I$ is an index set with
cardinality less than $\kappa$. Suppose $N_{0}\in\mathcal{K}^{<}\left(M\right)$
and $\gamma^{0}=\left(g_{i}^{0}:i\in I\right)$ is a sequence of automorphisms
of $N_{0}$. We say $\left(N_{0},\gamma^{0}\right)$ is a \emph{strong
amalgamation base} if for all $N_{1},N_{2}\in\mathcal{K}^{<}\left(M\right)$
with $N_{0}\prec_{\mathcal{K}}N_{1},N_{2}$; all $\gamma^{1}$ and
$\gamma^{2}$ sequences of automorphisms of $N_{1}$ and $N_{2}$
that extend $\gamma\text{\textsuperscript{0}}$; respectively, are
compatible over $N_{0}$. We say $\mathcal{K}^{<}\left(M\right)$
is a \emph{strong amalgamation class} if $\left(N,\gamma\right)$
is a strong amalgamation base for every $N\in$ $\mathcal{K}^{<}\left(M\right)$
and $\gamma=\left(g_{i}:i\in I\right)$ a sequence of automorphisms
of $N$.
\end{defn}
\begin{rem*}
As we see next, this notion guarantees the existence of generic sequences
of automorphisms. Its importance lies in the fact of being able to
``amalgamate coherently'' sequences of isomorphisms, while \emph{also}
doing an amalgam of the domains!
\end{rem*}
\begin{lem}
\label{lem:ex-gen} Suppose $\mathcal{K}^{<}\left(M\right)$ is a
strong amalgamation class. Then generic sequences of automorphisms
with index set of arbitrary cardinality less than $\kappa$ exist.
\end{lem}
\begin{proof}
We write the proof for the case when $I\text{ is a singleton}$. The
proof for the general case (arbitrary $I$) is not essentially different,
but is more cumbersome; therefore we address here only the case when
the sequence consists of a single automorphism.

We will build a generic automorphism $f$ as the union of a tower
of automorphisms of some elements of $\mathcal{K}^{<}(M)$. Let $M_{0}\in\mathcal{K}^{<}\left(M\right)$
and $f_{0}$ be an arbitrary automorphism of $M_{0}$. Since $M$
is homogeneous, we know that $\mathcal{O}_{f_{0}}^{M}\neq\emptyset$.
Now let $E:=\left\{ \left(N,e\right):M_{0}\prec_{\mathcal{K}}N\in\mathcal{K}^{<}\left(M\right)\mbox{ and }e\in\Aut\left(N\right)\mbox{ with }f_{0}\leqslant e\right\} $.
Since $\kappa^{<\kappa}=\kappa$, using Lemma \ref{lem:ep}, we have
that $\left|E\right|\leq\kappa$. We enumerate $E$ as $\left\{ \left(N_{i},e_{i}\right)\left|i<\kappa\right.\right\} $.
We now build a sequence of pairs $\left(M_{i},f_{i}\right)$ for $i<\kappa$
such that $\left(M_{i},f_{i}\right)\in E$, $M_{i}\prec_{\mathcal{K}}M_{j}$
and $f_{i}\leq f_{j}$ for $i\leq j<\kappa$. Moreover, we build the
sequence $\left(M_{i}\right)$ in such a way that $M=\bigcup_{i\in\kappa}M_{i}$
(for this, we enumerate $M$ as $\left\{ a_{i}\mid i<\kappa\right\} $
and demand that $a_{i}\in M_{i+1}$ for every $i<\kappa$).
\begin{itemize}
\item If $i$ is a limit ordinal then let $M_{i}=\bigcup_{\beta<i}M_{\beta}$
and $f_{i}=\bigcup_{\beta<i}f_{\beta}$; note that the cardinal assumption
$\kappa^{<\kappa}=\kappa$ implies that $M_{i}\in\mathcal{K}^{<}\left(M\right)$. 
\item Suppose $i$ is successor and $i=j+1$. Consider $\left(N_{j},e_{j}\right)\in E$.
Since $\mathcal{K^{<}}\left(M\right)$ is a strong amalgamation class,
there exist $M_{i}\in\mathcal{K^{<}}\left(M\right)$ and $f_{i}\in\Aut\left(M_{i}\right)$
such that $M_{j}\prec_{\mathcal{K}}M_{i}$, $f_{j}\leqslant f_{i}$,
$\alpha\in\Aut_{M_{j}}\left(M\right)$ and $e_{j}^{\alpha}:=\alpha\circ e_{j}\circ\alpha^{-1}\upharpoonright\alpha\left[N_{j}\right]\leqslant f_{i}$. 
\end{itemize}
Then let $f:=\bigcup_{i\in\kappa}f_{i}$ .

We claim $f$ is a generic automorphism. There are two cases:
\begin{enumerate}
\item Suppose $M_{0}\subseteq N_{1}$,$N\in\mathcal{K^{<}}\left(M\right)$
and $g\in\Aut\left(M\right)$ such that $N_{1}\prec_{\mathcal{K}}N$,
$g\upharpoonright N\in\Aut\left(N\right)$ and $g\upharpoonright N_{1}=f\upharpoonright N_{1}$.
Note that because of the coherence axiom of AECs $M_{0}\prec_{\mathcal{K}}N_{1}$.
Moreover, by Lemma $\ref{lem:ep}$ there is $N'\in\mathcal{K}^{<}\left(M\right)$
such that $N_{1}\prec_{\mathcal{K}}N'$ and $f':=f\upharpoonright N'\in\Aut\left(N'\right)$.
Let $i$ be the smallest index $<\kappa$ such that $N'\subseteq M_{i}$.
Since $N_{1}$ is an strong amalgamation base, there is $N_{2}\in\mathcal{K^{<}}\left(M\right)$
and $e_{2}\in\Aut\left(N_{2}\right)$ such that $M_{i},\alpha\left[N\right]\prec_{\mathcal{K}}N_{2}$,
$f'\leqslant e_{2}$, $\alpha\in\Aut_{N_{1}}\left(M\right)$ and $\alpha\circ g\circ\alpha^{-1}\leqslant e_{2}$.
It is clear that $\left(N_{2},e_{2}\right)\in E$ , then the construction
of $f$ guarantees that there is $j>i$ such that $\left(N_{2},e_{2}\right)$
embeds into $\left(M_{j},f_{j}\right)$ that fixes $M_{i}$. Hence,
and there exists $N^{*}\in\mathcal{K}^{<}\left(M\right)$ and $\alpha\in\Aut_{N_{1}}\left(M\right)$
such that $\alpha\left[N\right]=N^{*}$ and $f$ extends $\alpha\circ g\circ\alpha^{-1}\upharpoonright N^{*}$. 
\item Suppose $N,N'\in\mathcal{K^{<}}\left(M\right)$ and $g\in\Aut\left(M\right)$
are such that $N'\prec_{\mathcal{K}}N$, $g\upharpoonright N\in\Aut\left(N\right)$
, $g\upharpoonright N'=f\upharpoonright N'$ and $N'\cap M_{0}\neq M_{0}$.
Using Lemma \ref{lem:ep}, there is $N_{1}\in\mathcal{K^{<}}\left(M\right)$
such that $N\cup M_{0}\subseteq N_{1}$ and $f\upharpoonright N_{1}\in\Aut\left(N_{1}\right)$.
Again from the coherence axiom of AECs it follows that $N,M_{0}\prec_{\mathcal{K}}N_{1}$.
Since $\mathcal{K}^{<}\left(M\right)$ is a strong amalgamation class
and $\left(N,f\upharpoonright N\right)$ is a strong amalgamation
base, $f\upharpoonright N_{1}$ and $g\upharpoonright N'$ are compatible
over $N$. Let $\alpha\in\Aut_{N}\left(M\right)$, $N_{2}\in\mathcal{K}^{<}\left(M\right)$
and $h\in\Aut\left(N_{2}\right)$ such that $N_{1}\prec_{\mathcal{K}}N_{2}$
and $h$ extends $f\upharpoonright N_{1}\cup g^{\alpha}\upharpoonright N'$.
Note that $\alpha\left[N_{1}\right]$ (and hence $\alpha\left[N\right]$)
is strongly embedded in $N_{2}$ (again follows from coherence axiom
in AEC's). Now $\left(N_{2},h\right)$ is an element of $E$. Then
we can again reason as in Case 1.
\end{enumerate}
\end{proof}
\begin{lem}
Fix an index set $I$ with $\left|I\right|<\kappa$ and let $\mathcal{F}$
be the set of all $I$-sequences of generic automorphisms. Then $\mathcal{F}$
is a dense subset of $\Aut\left(M\right)^{I}$, with the product topology. 
\end{lem}
\begin{proof}
For the sake of readability, we write the proof again for the case
$\left|I\right|=1$. Let $\Aut_{A}\left(M\right)$ be a basic open
set with $A\in\mathcal{C}$. Let $g\in\Aut_{A}\left(M\right)$; then,
just as in the proof of Lemma \ref{lem:ex-gen} (choosing $f_{0}:=g\upharpoonright A$)
we can find a generic automorphism $f$ such that $f_{0}\leqslant f$.
Therefore $f\in\Aut_{A}\left(M\right)$ and $\mathcal{F}\cap\Aut_{A}\left(M\right)$
is non-empty.
\end{proof}
\begin{rem}
In Definitions \ref{def:comp}, \ref{def:gen} and \ref{Strong-Amalgamation-Base}
we considered the very natural case where the elements are from $\mathcal{K}^{<}\left(M\right)$
and it is enough for us to prove the results of next sections. However,
these notions can be considered for the elements from the bigger set
$\mathcal{C}$; the required definitions for elements of $\mathcal{C}$
have already been provided in Section \ref{sec:setting}.
\end{rem}

\subsection{Many generic sequences of automorphisms}

We now prove that there are many different generic sequences of automorphisms
\emph{relative to the size} of $\Aut\left(M\right)$, in the following
specific sense: 
\begin{itemize}
\item If $\left|\Aut\left(M\right)\right|\geq\kappa$, then for any family
$\left(M_{i}\right)_{i\in I}$ of ($\leq$$\kappa$)-many models in
$\mathcal{K}^{<}\left(M\right)$ where $\mathcal{K}^{<}\left(M\right)$
is a strong amalgamation base, we obtain a corresponding family $\left(h_{i}\right)_{i\in I}$
of automorphisms of $M$, each $h_{i}$ fixing pointwise $M_{i}$,
and such that the family $\left(h_{i}\right)_{i\in I}$ can be used
to build a generic sequence of automorphisms in $2^{\kappa}$-many
different ways (the technical details of this are the content of the
next lemma).
\item We therefore obtain $2^{\kappa}$-many generic sequences of automorphisms
of length $\kappa$. Moreover, this implies that if $\left|\Aut\left(M\right)\right|\geq\kappa$
then $\left|\Aut\left(M\right)\right|=2^{\kappa}$. Our approach in
this part has the general structure of Lascar-Shelah \cite[Lemma 9]{MR1204064}
but making explicit the use of strong amalgamation bases.
\item Of course, in the other extreme case when the model $M$ is rigid\footnote{For instance, in the case where the model $M$ is a cardinal and strong
elementary submodels of $M$ are just initial segments.}, all generic sequences of automorphisms must consist of the identity.
\end{itemize}
\begin{lem}
\label{lem:matrix} Suppose $\left(g_{i,j}:i\in I,j\in J\right)$
is a matrix of automorphisms of $M$ and $\left|I\right|=\left|J\right|\leq\kappa$.
Fix $\left(M_{i}:i\in I\right)$ a sequence of elements of $\mathcal{K^{<}}\left(M\right)$
where $\mathcal{K}^{<}\left(M\right)$ is a strong amalgamation base.
Then there exists $\left(h_{i}:i\in I\right)$ such that
\begin{enumerate}
\item $h_{i}\in\Aut_{M_{i}}\left(M\right)$ for all $i\in I$;
\item $\left(h_{i}\circ g_{i,\delta\left(i\right)}:i\in I\right)$ is a
generic sequence, for all injective functions $\delta:I\rightarrow J$.
\end{enumerate}
\end{lem}
\begin{proof}
Consider the following set 
\begin{eqnarray*}
\mathcal{Y} & = & \{\left(I_{0},N_{0},N_{1},\delta,\left(k_{i}:i\in I_{0}\right)\right):I_{0}\subseteq I,\left|I_{0}\right|<\kappa,\delta\mbox{ is an injective function from }I_{0}\mbox{ into }J\mbox{ with }\\
 &  & N_{0}\preccurlyeq N_{1}\in\mathcal{K}^{<}\left(M\right),k_{i}\in\Aut\left(N_{1}\right)\mbox{ and }k_{i}\upharpoonright N_{0}\in\Aut\left(N_{0}\right)\}.
\end{eqnarray*}
Note that $\left|\mathcal{Y}\right|=\kappa$. Fix an enumeration $\left(y_{\alpha}:\alpha<\kappa\right)$
for $\mathcal{Y}$. We define by induction on $\alpha<\kappa$ a family
$\left(M_{\alpha}^{i}:i\in I\right)$ of elements of $\mathcal{K}^{<}\left(M\right)$
and a family $\left(h_{\alpha}^{i}:i\in I\right)$, where $h_{\alpha}^{i}\in\Aut\left(M_{\alpha}^{i}\right)$
such that
\begin{enumerate}
\item $M_{0}^{i}=M_{i}$ and $h_{0}^{i}=id_{M_{i}}$ for all $i\in I$;
\item The functions $\left(\alpha\mapsto M_{\alpha}^{i}\right)$ and $\left(\alpha\mapsto h_{\alpha}^{i}\right)$
are increasing and continuous, for all $i\in I$;
\item $\bigcup_{\alpha<\kappa}M_{\alpha}^{i}=M$ and $h_{i}:=\bigcup_{\alpha<\kappa}h_{\alpha}^{i}\in\Aut\left(M\right)$
for each $i\in I$;
\item $\left(h_{i}\circ g_{i,\delta\left(i\right)}:i\in I\right)$ is a
generic sequence for all injective functions $\delta$ from $I$ to
$J$.
\end{enumerate}
Assume $y_{\alpha}=\left(I_{0},N_{0},N_{1},\delta,\left(k_{i}:i\in I_{0}\right)\right)\in\mathcal{Y}$.
Using Lemma $\ref{lem:ep}$ and Corollary $\ref{cor:ext}$ one can
show that there is a model $M'\in\mathcal{K}^{<}\left(M\right)$ such
that $N_{0}\prec_{\mathcal{K}}M'$
\begin{itemize}
\item $M_{\alpha}^{i}\prec_{\mathcal{K}}M'$ for all $i\in I_{0}$;
\item $g_{i,\delta\left(i\right)}\upharpoonright M'\in\Aut\left(M'\right)$
for all $i\in I_{0}$;
\item $h_{\alpha}^{i}$ extends to an automorphism $m_{\alpha}^{i}\text{\ensuremath{\in\Aut\left(M'\right)}}$
for each $i\in I_{0}$. 
\end{itemize}
If $\left(m_{\alpha}^{i}\circ g_{i,\delta\left(i\right)}:i\in I_{0}\right)$
and $\left(k_{i}:i\in I_{0}\right)$ are not compatible over any $\prec_{\mathcal{K}}$-submodel
of $N_{0}$ (i.e. there is no model and no $\prec_{\mathcal{K}}$-embedding
to that model such that we can amalgamate the automorphisms over any
$\prec_{\mathcal{K}}$-submodel of $N_{0}$), then define $M_{\alpha+1}^{i}:=M'$
and $h_{\alpha+1}^{i}:=m_{\alpha}^{i}$ for all $i\in I_{0}$. Define
$M_{\alpha+1}^{i}=M_{\alpha}^{i}$ and $h_{\alpha+1}^{i}=h_{\alpha}^{i}$
for $i\notin I_{0}$. 

Suppose, without loss of generality, $\left(m_{\alpha}^{i}\circ g_{i,\delta\left(i\right)}:i\in I_{0}\right)$
and $\left(k_{i}:i\in I_{0}\right)$ are compatible over $N_{0}$.
Since $\mathcal{K}^{<}\left(M\right)$ is a strong amalgamation class
we can find $M_{\alpha+1}^{i}\succ_{\mathcal{K}}M'$ and $h_{\alpha+1}^{i}\in\Aut\left(M_{\alpha+1}^{i}\right)$
such that there exists an $\prec_{\mathcal{K}}$-elementary map $f$
from $N_{1}$ into $M_{\alpha+1}^{i}$ which extends the identity
on $N_{0}$ in such a way that $h_{\alpha+1}^{i}\circ g_{i,\delta\left(i\right)}\upharpoonright M_{\alpha+1}^{i}$
extends $f\circ k_{i}\circ f^{-1}$. If $i\notin I_{0}$, then define
$M_{\alpha+1}^{i}=M_{\alpha}^{i}$ and $h_{\alpha+1}^{i}=h_{\alpha}^{i}$. 
\end{proof}
Now for the case where $\left|\Aut\left(M\right)\right|\geq\kappa$
we choose $\left(g_{i,j}:i\in I,j\in J\right)$ to be a $\kappa\times\kappa$-matrix
of distinct automorphisms of $M$. Then we get $2^{\kappa}$-many
generic sequences of automorphisms of length $\kappa$. 

\section{Proof of SIP for strong amalgamation classes}

\label{sec:main-thm}

In this section we prove our main theorem.
\begin{thm}
\label{thm:main} Let $M$ be a homogeneous model in an abstract elementary
class $\left(K,\preccurlyeq_{\mathcal{K}}\right)$ such that \textup{$\left|M\right|=\kappa>\mbox{LS}\left(K\right)$
}\textup{\emph{and}}\textup{ $\kappa^{<\kappa}=\kappa$. }\textup{\emph{Furthermore,
assume that }}\emph{$\mathcal{K}^{<}\left(M\right)$ }is a\emph{ }strong
amalgamation clas\textup{\emph{s. }}Consider the group\textup{\emph{
$Aut(M)$}}\textup{ }with the topology given by $\mathfrak{T}^{\mbox{cl}}$,
and let $H\leqslant\Aut\left(M\right)$ be such that $\left[\Aut\left(M\right):H\right]\leq\kappa$.
\uline{Then}, $H$ is an open subgroup of $Aut(M)$; i.e., there
exists $A\in\mathcal{K}^{<}\left(M\right)$ such that $\Aut_{A}\left(M\right)\leqslant H$.

In other words, $\left(\Aut\left(M\right),\mathfrak{T}^{\mbox{cl}}\right)$
has the small index property. 
\end{thm}
We need only deal with the case $\Aut\left(M\right)$ is \emph{rich}\footnote{\emph{With an argument similar to the first order countable case,
one can show that if the identity is not isolated, then $\Aut\left(M\right)$
is a perfect complete topological space.}}, that is $\left|\text{\ensuremath{\Aut\left(M\right)}}\right|=2^{\kappa}$
(otherwise, the identity would be isolated and therefore the automorphism
group would automatically satisfy the theorem, as in that case all
subgroups would be open). 

Suppose $H$ is a subgroup of $\Aut\left(M\right)$ with small index
(i.e. not bigger that $\kappa$). Toward a contradiction suppose $H$
is not open.
\begin{prop}
\label{prop:many-aut} (Similar to Proposition 10 in \cite{MR1204064})
There exists a generic sequence $\gamma=\left(g_{i}:i\in I\right)$
such that
\begin{enumerate}
\item the set $\left\{ i\in I:g_{i}\upharpoonright M_{0}=h\mbox{ and }g_{i}\notin H\right\} $
has cardinality $\kappa$ for all $M_{0}\in\mathcal{K^{<}}\left(M\right)$
and $h\in\Aut\left(M_{0}\right)$;
\item the set $\left\{ i\in I:g_{i}\in H\right\} $ has cardinality $\kappa$.
\end{enumerate}
\end{prop}
\begin{proof}
Consider the following set $\mathcal{X}=\left\{ \left(M_{0},f\right):M_{0}\in\mathcal{K^{<}}\left(M\right)\mbox{ and }f\in\Aut\left(M_{0}\right)\right\} $.
It is clear that $\left|\mathcal{X}\right|=\kappa$. 
\begin{itemize}
\item Consider $I_{0}$ of cardinality $\kappa$ and a sequence $\left(\left(M_{i},f_{i}\right):i\in I_{0}\right)$
of elements of $\mathcal{X}$ such that the set $\left\{ i\in I_{0}:\left(M_{i},f_{i}\right)=\left(M_{0},f\right)\right\} $
has cardinality $\kappa$ for all $\left(M_{0},f\right)\in\mathcal{X}$.
\item Let $I_{1}$ be a set of cardinality $\kappa$ disjoint from $I_{0}$
and let $I=I_{0}\cup I_{1}$.
\item Finally, let $J$ be any set of cardinality $\kappa$.
\end{itemize}
For each $i\in I$ and $j\in J$ define $g_{i,j}\in\Aut\left(M\right)$
such that the following hold:
\begin{enumerate}
\item $g_{i,j}\upharpoonright M_{i}=f_{i}$ for all $i\in I_{0}$. Moreover,
the set $\left\{ g_{i,j}:j\in J\right\} $ where $i\in I_{0}$ meets
at least two classes modulo $H$. This is always possible since we
assumed $H$ is not open (and hence none of its classes contain a
non-empty open set).
\item The set $\left\{ g_{i,j}\in J\right\} $ meets all classes modulo
$H$ if $i\in I_{1}$. This is possible because the index of $H$
in $\Aut\left(M\right)$ is small (i.e. not bigger that $\kappa$).
\end{enumerate}
Now let $\left(M_{i}:i\in I\right)$ a sequence of elements of $\mathcal{K^{<}}\left(M\right)$
where $M_{i}$ is arbitrary for $i\in I_{1}$. By Lemma \ref{lem:matrix}
there is a family $\left(h_{i}:i\in I\right)$ such that satisfying
conditions (1) and (2) of this lemma. Then choose a bijective function
$\delta:I\rightarrow J$ such that:
\begin{itemize}
\item $g_{i,\delta\left(i\right)}$ is not in the class of $h_{i}^{-1}$
for $i\in I_{0}$ (i.e. $g_{i,\delta\left(i\right)}\notin h_{i}^{-1}H$
guaranteed by Condition (1) above);
\item $g_{i,\delta\left(i\right)}$ is in the class of $h_{i}^{-1}$ for
$i\in I_{1}$ (i.e. $g_{i,\delta\left(i\right)}\in h_{i}^{-1}H$ guaranteed
by Condition (2) above).
\end{itemize}
Then the sequence $\gamma:=\left(h_{i}\circ g_{i,\delta\left(i\right)}:i\in I\right)$
is a generic sequence and satisfies the requirement.
\end{proof}

\subsection{Construction of the tree}

Let $H$ be as before: a subgroup of small index and we assume, toward
a contradiction, that $H$ is not open. Let $\mathcal{S}=2^{<\kappa}$
be the set of sequences of $0$ and $1$ of length less than $\kappa$,
and $\mathcal{S}^{*}=\left\{ s\in\mathcal{S}:\mbox{ the length of \ensuremath{s} is successor}\right\} $.
Let $\gamma$ be the generic sequence that is obtained from Proposition
\ref{prop:many-aut}. Fix an enumeration $\left(a_{\alpha}:\alpha<\kappa\right)$
of elements of $M$. 

We construct by induction on $s\in\mathcal{S}$, a model $M_{s}\in\mathcal{K}^{<}\left(M\right)$,
an automorphism $g_{s}\in\Aut\left(M_{s}\right)$, and if $s\in\mathcal{S}^{*}$,
automorphisms $h_{s}$ and $k_{s}$ in $\Aut_{\left\{ M_{s}\right\} }\left(M\right)$
in such a way that the following conditions are satisfied:
\begin{enumerate}
\item $h_{s,0}\in H$ and $h_{s,1}\notin H$ for all $s\in\mathcal{S}^{*}$;
\item $k_{s,0}=k_{s,1}$ for all $s\in\mathcal{S}^{*}$;
\item $h_{t}\left[M_{s}\right]=M_{s}$ (i.e. $h_{t}\in\Aut_{\{M_{s}\}}\left(M\right)$)
for $s\in\mathcal{S}$ and all $t\in\mathcal{S}^{*}$ with $t\leqslant s$;
\item $g_{s}\circ\left(h_{t}\upharpoonright M_{s}\right)\circ g_{s}^{-1}=k_{t}\upharpoonright M_{s}$
for $s\in\mathcal{S}$ and all $t\in\mathcal{S}^{*}$ with $t\leqslant s$;
\item $a_{\beta}\in M_{s}$ for $s\in\mathcal{S}$ and $\beta<\mbox{length}\left(s\right)$;
\item $\left(h_{t}:t\leqslant s,t\in\mathcal{S}^{*}\right)$ and $\left(k_{t}:t\leqslant s,t\in\mathcal{S}^{*}\right)$
are sequences elements from $\gamma$, for $s\in\mathcal{S}$ (and
they are generic as well).
\end{enumerate}
For $s=\emptyset$, define $M_{s}$ to be an arbitrary element of
$\mathcal{K}^{<}\left(M\right)$ and $g_{s}\in\mbox{id}\left(M_{s}\right)$.
For limit step there is no problem. Suppose everything has been defined
up to step $s$. Write $\mathcal{F}$ for the set of all the automorphisms
of the generic sequence $\gamma$. First choose $h_{s,0}\in\mathcal{F}\cap H$
not in $\left\{ h_{t}:t\leqslant s,t\in\mathcal{S}^{*}\right\} $.
Extend $g_{s}$ to $g\in\Aut\left(M\right)$ such that $g\circ h_{t}\circ g^{-1}=k_{t}$
for all $t\in\mathcal{S}^{*}$ with $t\leqslant s$ in the following
manner: First extend $g_{s}$ to some $g'\in\Aut\left(M\right)$.
The two families $\left\{ g'\circ h_{t}\circ g'^{-1}:t\leqslant s,t\in\mathcal{S}^{*}\right\} $
and $\left\{ k_{t}:t\leqslant s,t\in\mathcal{S}^{*}\right\} $ are
generic and they agree on $M_{s}$. Hence, we can find $g''\in\Aut_{M_{s}}\left(M\right)$
such that $k_{t}=g''\circ g'\circ h_{t}\circ g'^{-1}\circ g''^{-1}$
for all $t\in\mathcal{S}^{*}$ with $t\leqslant s$. Then let $g=g''\circ g'$. 

Using Lemma Corollary $\ref{cor:ext}$ we can find $M_{s,0}\in\text{\ensuremath{\mathcal{K}}}^{<}\left(M\right)$
in such a way that:
\begin{enumerate}
\item $M_{s,0}$ contains $M_{s}$ and $a_{\alpha}$, where $\alpha=\mbox{length}\left(s\right)$;
\item $h_{t}\left[M_{s,0}\right]=M_{s,0}$ for all $t\leqslant\left(s,0\right)$
and $g\left[M_{s,0}\right]=M_{s,0}$.
\end{enumerate}
Set $M_{s,1}=M_{s,0}$, $g_{s,0}=g_{s,1}=g\upharpoonright M_{s,0}$
and $h_{s,1}$ an element of $\mathcal{F}$ extending $h_{s,0}\upharpoonright M_{s,0}$,
not in $H$ and not in $\left\{ h_{t}:t\leqslant s,t\in\mathcal{S}^{*}\right\} $,
and $k_{s,0}=k_{s,1}$ an element of $\mathcal{F}$ extending $g_{s,0}\circ\left(h_{s,0}\upharpoonright M_{s,0}\right)\circ g_{s,0}^{-1}$
not in $\left\{ k_{t}:t\leqslant s,t\in\mathcal{S}^{*}\right\} $. 

For each $\sigma\in2^{\kappa}$, let $g_{\sigma}=\bigcup_{\delta<\sigma}g_{\delta}$.
Then $g_{\sigma}\in\Aut\left(M\right)$ and moreover for all $t<\sigma$
and $t\in\mathcal{S}^{*}$, $g_{\sigma}\circ h_{t}\circ g_{\sigma}^{-1}=k_{t}$.
Assume $\tau$ and $\sigma$ are two distinct elements of $2^{\kappa}$;
let $s$ be their largest common initial segment and assume, with
out loss of generality, that $\left(s,0\right)<\sigma$ and $\left(s,1\right)<\tau$.
Then $g_{\sigma}\circ h_{s,0}\circ g_{\sigma}^{-1}=k_{s,0}=k_{s,1}=g_{\tau}\circ h_{s,1}\circ g_{\tau}^{-1}$;
thus $h_{s,0}=g_{\sigma}^{-1}\circ g_{\tau}\circ h_{s,1}\circ g_{\tau}^{-1}\circ g_{\sigma}$.
Since $h_{s,0}\in H$ and $h_{s,1}\notin H$, $g_{\sigma}^{-1}\circ g_{\tau}\notin H$
hence the index of $H$ in $\Aut\left(M\right)$ is $2^{\kappa}$;
a contradiction.

\section{The SIP for the Zilber field, covers and modular invariants}

This section consists of applications of Theorem \ref{thm:main}.
We first provide a general setting (\emph{quasiminimal pregeometry
classes}) that both satisfies enough conditions for the theorem to
apply and includes several interesting cases.

\subsection{The SIP and quasiminimal pregeometry classes\label{subsec:The-SIP-and}}

In order to construct an algebraically closed field with the \emph{Schanuel
property} categorical in all uncountable cardinals, Zilber introduced
the notion of ``quasiminimal excellent'' classes (see \cite{MR2102856})
- a combination of Shelah's excellent classes with a weaker variant
of strong minimality. Later, the concept was studied in isolation
and two major simplifications were brought about: first of all, a
notion of quasiminimal classes is enough to imply excellence (see
\cite{BHHKK}) and a similar notion of quasiminimal classes directly
implies categoricity (see \cite{LH}). In this chapter we prove, as
an application of our Theorem $\ref{thm:main}$ that the SIP for homogeneous
models holds in the quasiminimal classes setting. Then, we conclude
that the automorphism group of Zilber field has the small index property.

We use the setting from \cite{LH}. For the sake of completeness (and
clarification) we restate the axioms here:
\begin{defn}
Let $\mathfrak{L}$ be a language. A\emph{ quasiminimal pregeometry}
\emph{class} $\mathcal{Q}$ is a class of pairs $\left\langle H,\mbox{cl}_{H}\right\rangle $
where $H$ is an $\mathfrak{L}$-structure and $\mbox{cl}_{H}$ is
a pregeometry operator on $H$ such that the following conditions
hold:

\begin{enumerate}
\item \textbf{Closure under isomorphisms:} If $\left\langle H,\mbox{cl}_{H}\right\rangle \in\mathcal{Q}$
and $H'$ is an $\mathfrak{L}$-structure and $f:H\rightarrow H'$
is an isomorphism then $\left\langle H',\mbox{cl}_{H'}\right\rangle \in\mathcal{Q}$
where $\mbox{cl}_{H'}:=f\left(\mbox{cl}_{H}\left(f^{-1}\left(X'\right)\right)\right)$
for $X'\subseteq H'$.
\item \textbf{Quantifier free theory:} The empty function is a partial embedding
between any two structures of the class $\mathcal{Q}$.
\item \textbf{Countable closure: }For each $\left\langle H,\mbox{cl}_{H}\right\rangle \in\mathcal{Q}$,
the closure of any finite set is countable.
\item \textbf{Relativization: }If $\left\langle H,\mbox{cl}_{H}\right\rangle \in\mathcal{Q}$
and $X\subseteq H$, then $\left\langle \mbox{cl}_{H}\left(X\right),\mbox{cl}_{H}\upharpoonright\mbox{cl}_{H}\left(X\right)\right\rangle \in\mathcal{Q}$.
\item \textbf{Closure coherence: }If $\left\langle H,\mbox{cl}_{H}\right\rangle ,\left\langle H',\mbox{cl}_{H'}\right\rangle \in\mathcal{Q}$,
$X\subseteq H$, $y\in H$ and $f:H\rightarrow H'$ is a partial embedding
defined on $X\cup\left\{ y\right\} $, then $y\in\mbox{cl}_{H}\left(X\right)$
if and only if $f\left(y\right)\in\mbox{cl}_{H'}\left(f\left(X\right)\right)$.
\item \textbf{Homogeneity over countable models / Uniqueness of generic
types:} Let $\left\langle H,\mbox{cl}_{H}\right\rangle ,\left\langle H',\mbox{cl}_{H'}\right\rangle \in\mathcal{Q}$
and suppose $C\subseteq H,C'\subseteq H'$ are countable closed sets
and $g:C\rightarrow C'$ is an isomorphism:

\begin{enumerate}
\item If $x\in H\backslash C$ and $x'\in H'\backslash C'$, then $g\cup\left\{ \left(x,x'\right)\right\} $
is a partial embedding.
\item If $g\cup f:H\rightharpoonup H'$ is a partial embedding such that
$X=dom\left(f\right)$ is finite and $y\in\mbox{cl}_{H}\left(X\cup C\right)$,
then there is $y'\in H'$ such that $g\cup f\cup\left\{ \left(y,y'\right)\right\} $
is a partial embedding.
\end{enumerate}
\end{enumerate}
Given a structure $H\in\mathcal{Q}$ and a substructure $G\subseteq H$
also in $\mathcal{Q}$, denote $G\preccurlyeq H$ when $G$ is closed
in $H$.
\end{defn}
\begin{prop}
\label{prop:clo} (Proposition 4 in \cite{LH}) Let $H,H'\in\mathcal{Q}$
such that $H\preccurlyeq H'$ and $X\subseteq H$. Then $\mbox{cl}_{H}\left(X\right)=\mbox{cl}_{H'}\left(X\right)$.
\end{prop}
\begin{lem}
Fix $M\in\mathcal{Q}$ where $\mathcal{Q}$ is a quasiminimal pregeometry
class. Then $\mbox{cl}_{M}\left(A\right)=\mbox{cl}^{M}\left(A\right)$
(the intersection of all closed sets containing $A$) for all $A\subseteq M$.
\end{lem}
\begin{proof}
Let $A\subseteq M$ then $\mbox{cl}_{M}\left(A\right)\in\mathcal{Q}$
(by point 4 of the definition of $\mathcal{Q}$). Therefore $\mbox{cl}^{M}\left(A\right)\subseteq\mbox{cl}_{M}\left(A\right)$.
The other direction is also simple: suppose $a\in\mbox{cl}_{M}\left(A\right)\backslash\mbox{cl}^{M}\left(A\right)$,
then there is $H\in\mathcal{Q}$ such that $A\subseteq\mbox{cl}^{M}\left(A\right)\subseteq H$,
$H\preccurlyeq M$ and $a\notin H$. Note that $\mbox{cl}_{H}\left(A\right)\subseteq H$
does not contain $a$; but on the other hand from Proposition \ref{prop:clo}
it follows that $\mbox{cl}_{H}\left(A\right)=\mbox{cl}_{M}\left(A\right)$
and $a\in\mbox{cl}_{M}\left(A\right)=\mbox{cl}_{H}\left(A\right)$
which is a contradiction.
\end{proof}
Recall that in \cite{LH}, Haykazyan derives uncountable categoricity
from the quasiminimal pregeometry axioms. It is worth mentioning that
in November 2016, during corrections to an earlier version of this
paper, Sebastien Vasey has posted an article in ArXiv \cite{1611.07380}
where he proves that in AECs admitting intersections (i.e. with a
notion of closure defined like our $cl^{M}$), the exchange axiom
follows from the other axioms of a quasiminimal AEC.
\begin{thm}
\label{thm:2} Suppose $\mathcal{Q}$ is a quasiminimal pregeometry
class. Let $M\in\mathcal{Q}$ be the model of size $\aleph_{1}$ and
let $\mathcal{C}:=\left\{ \text{cl}_{M}\left(A\right):A\subseteq M,\left|A\right|<\left|M\right|\right\} $.
Then the class $\mathcal{C}$ is a strong amalgamation class and hence,
$\Aut\left(M\right)$ has SIP.
\end{thm}
\begin{proof}
Let $N_{0}\in\mathcal{C}$ and $\gamma^{0}\in\Aut\left(N_{0}\right)$.
We want to show $\left(N_{0},\gamma^{0}\right)$ is a strong amalgamation
basis. Suppose $N_{1},N_{2}\in\mathcal{K}^{<}\left(M\right)$ are
such that $N_{0}\preccurlyeq N_{1},N_{2}$; and let $\gamma^{1}$
and $\gamma^{2}$ be sequences of automorphisms of $N_{1}$ and $N_{2}$
extending $\gamma^{0}$; respectively. We claim one can find $N'_{2}\in\mathcal{C}$
such that $N'_{2}\cong_{N_{0}}N_{2}$ and $N'_{2}\cap N_{1}=N_{0}$.
Let $B$ be a basis for $N_{2}$ (i.e. $\mbox{cl}_{M}\left(B\right)=N_{2}$).
Choose $B'$ an independent set in $M$ such that there is a bijection
between $B'$ and $B$, and moreover $B'\cap N_{1}=B_{0}=B\cap N_{2}$
where $\mbox{cl}_{M}\left(B_{0}\right)=N_{0}$. Now let $N'_{2}=\mbox{cl}_{M}\left(B'\right)$.
Note that we can choose $B'$ in such a way that $N'_{2}\cap N_{1}=N_{0}$.
By Theorem 16 in \cite{LH} there is an isomorphism $\alpha$ between
$N_{2}$ and $N_{2}'$ and therefore $N'_{2}\cong_{N_{0}}N_{2}$.
Now $\gamma^{1}\cup\left(\alpha\circ\gamma^{2}\circ\alpha^{-1}\right)$
is an isomorphism of $N_{1}\cup\alpha\left[N_{2}\right]$ to itself.
Then by Proposition 14 in \cite{LH} extends $\gamma^{1}\cup\left(\alpha\circ\gamma^{2}\circ\alpha^{-1}\right)$
extends to an automorphism of $N_{3}:=\mbox{cl}_{M}\left(N_{1}\cup\alpha\left[N_{2}\right]\right)$
and then to an automorphism of $M$. Therefore it follows that $\gamma^{1}\cup\left(\alpha\circ\gamma^{2}\circ\alpha^{-1}\right)$
extends to an automorphism of $N_{3}$ where $N_{3}\in\mathcal{C}$
and $N_{1},\alpha\left[N_{2}\right]\preccurlyeq N_{3}$. Hence, $\left(N_{0},\gamma^{0}\right)$
is an amalgamation base. 
\end{proof}
\begin{cor}
Let $B$ be the Zilber field. Then $\Aut\left(B\right)$ has SIP for
the topology where basic open sets around the identity are given by
stabilizers of sets of size less than continuum.
\end{cor}
It is an interesting question to check whether by investigating the
automorphism group of the Zilber field (which has the Schanuel property)
we would be able to detect any difference between the Zilber field
and $\left(\mathbb{C},+,\cdot,\mbox{exp}\right)$. From \cite{MR1204064}
it follows that the complex numbers (without $exp$) have the SIP.
The construction of the Zilber field, using the quasiminimal excellence
setting, allowed us to prove SIP for its automorphism group. Then
the question of whether or not we are able to detect similar properties
about the automorphism group of $\left(\mathbb{C},+,\cdot,\mbox{exp}\right)$
seems plausible.

\subsection{More $L_{\omega_{1},\omega}$-sentences and the SIP (other directions)}

Now that we have established the SIP for classes arising from quasiminimal
pregeometry classes (really, for their model of size $\aleph_{1}$)
and we have shown that being a quasiminimal pregeometry class is enough
to obtain that the class of small elementary submodels satisfies strong
amalgamation, a natural question is for which other classes axiomatizable
by $L_{\omega_{1},\omega}$-sentences can we guarantee the SIP (at
uncountable homogeneous models!). A natural class would be \emph{excellent
}$L_{\omega_{1},\omega}$-sentences that are not necessarily quasiminimal.
The study of the exact conditions that would guarantee that homogeneous
models of those classes have strong amalgamation are, however, the
subject of possible future work.

\subsection{Covers}

The fact that all quasiminimal pregeometry classes satisfy our framework
for SIP (Theorem \ref{thm:2}) provides a further source of examples.
In section \ref{subsec:The-SIP-and}, we analyzed the case of the
``Zilber field'' as a quasiminimal pregeometry class. This is indeed
one instance of a much more general phenomenon: some covers and some
modular invariants (like the classical $j$-function) have been shown
to be quasiminimal pregeometry classes and therefore their homogeneous
models $M$ have automorphism groups $Aut(M)$ that are SIP, by our
results. 

Variants of this example have been studied extensively by Zilber and
others (Baldwin, Bays, Caycedo, Kirby, Sustretov, etc.). The general
setup arises from a short exact sequence

\[
0\to K\overset{i}{\to}G\overset{f}{\to}H\to\text{1}
\]
where $G,H$ are groups (typically $G$ is an additive abelian group
and H is a multiplicative group) and the map $f:G\to H$ is a homomorphism
with kernel $K$. In particular, for this situation, Bays and Zilber
provide in \cite{BaysZilber} conditions for the quasiminimality of
the structure. Any structure axiomatized in $L_{\omega_{1},\omega}$
that satisfies the Bays-Zilber conditions for quasiminimality will
have the SIP.

The ``Zilber field'' of our previous section has the added complexity
that we are dealing with the field structure and not just with a homomorphism
from the additive group structure into the multiplicative group part.
A simplified structure often studied (see \cite{MR2159728}) in connection
with complex exponentiation is sometimes presented as a cover

\[
0\to2\pi i\mathcal{\mathbb{Z}}\overset{i}{\to}\left(\mathbb{C},+\right)\overset{exp}{\to}\left(\mathbb{C}^{\times},\cdot\right)\to1.
\]

\subsection{j-mappings}

Harris \cite{Geometric-Pillay} defines in an $L_{\omega_{1},\omega}$-axiomatization
of the classical ``modular invariant'', also called the ``$j$-mapping'',
an analytic function from the upper half plane $\mathbb{H}$ into
the complex numbers, 
\[
j:\mathbb{H}\to\mathbb{C}.
\]
 The $j$-mapping is a crucial component of analytic number theory.
It provides surprising connections between properties of extensions
of number fields and analytic properties of the mapping - the solution
to Hilbert's Twelfth problem for the characteristic zero and complex
case hinges on this.

Harris's model-theoretic analysis of $j$, developed in \cite{Harris},
produces a quasiminimal pregeometry and therefore is a new example
of a class with the SIP. We describe briefly the backbone of Harris's
analysis, and the connection to quasiminimality and the SIP.

Let $L$ be a language for two-sorted structures of the form 
\[
{\mathfrak{A}}=\langle\langle H;\{g_{i}\}_{i\in{\mathbb{N}}}\rangle,\langle F,+,\cdot,0,1\rangle,j:H\to F\rangle
\]
where $\langle F,+,\cdot,0,1\rangle$ is an algebraically closed field
of characteristic $0$, $\langle H;\{g_{i}\}_{i\in{\mathbb{N}}}\rangle$
is a set together with countably many unary function symbols (representing
the action of a countable group on the upper half plane), and $j:H\to F$.
Really, $j$ is a \textbf{cover} from the \emph{action} structure
into the field $\mathcal{\mathfrak{\mathbb{C}}}$. 

Let then 
\[
Th_{\omega_{1},\omega}(j):=Th(\mathbb{C}_{j})\cup\forall x\forall y(j(x)=j(y)\to\bigvee_{i<\omega}x=g_{i}(y))
\]
 where $\mathbb{C}_{j}$ is the ``standard $j$-model'' $(\mathbb{H},\langle\mathbb{\mathbb{C}},+,\cdot,0,1\rangle,j:\mathbb{H}\to\mathbb{C})$
and $Th(\mathbb{C}_{j})$ is its first order theory.

This captures all the first order theory of $j$ (not the analyticity!)
plus the fact that fibers are ``standard'' (``fibers are orbits'').

Harris analyzes this structure in order to establish (modulo serious
algebraic geometric results\footnote{In his proof of categoricity, Harris uses an instance of the so called
``adelic Mumford-Tate conjecture'' (proved by Serre) for products
of elliptic curves. The strategy to build an isomorphism between two
models $M$ and $M'$ partially consists in proving that the closure
given on the ``$H$ sort'' by $cl_{j}(A)=j^{-1}(acl(j(A)))$ provides
a quasiminimal pregeometry structure. Proving this is not trivial
at all: one needs to control the quantifier-free types of ``elliptic
curves'' as represented by the parameter in $H$ in terms of their
interactions with other such parameters.}) the categoricity of $Th_{\omega_{1},\omega}(j)+{\rm trdeg}(F)\geq\aleph_{0}$
in all infinite cardinalities. A crucial step in this work consists
in proving that the notion of closure given as
\[
cl_{j}(A):=j^{-1}(acl(j(A)))
\]

indeed defines a quasiminimal pregeometry structure.

Now, our main theorem implies (modulo Harris's quite long analysis)
that the ``automorphism group of $j$'', that is, the set of map-pairs
$\varphi_{H},\varphi_{F}$ such that

\[
\xymatrix{\mathcal{H}_{1}\ar[d]_{j_{1}}\ar[r]^{\varphi_{\mathcal{H}}} & \mathcal{H}_{2}\ar[d]^{j_{2}}\\
\mathcal{F}_{1}\ar[r]_{\varphi_{\mathcal{F}}} & \mathcal{F}_{2}
}
\]
commutes, has the SIP (for the topology where basic open sets around
the identity are stabilizers of sets of size less than continuum).

\selectlanguage{american}%
\bibliographystyle{acm}
\selectlanguage{english}%

\end{document}